\documentclass[11pt]{amsart}
\input{diagrams}

\usepackage{bibentry}
\usepackage{amsmath}
\usepackage{amsthm}
\usepackage{amssymb}
\usepackage{amsfonts}
\usepackage{amsxtra}
\usepackage{amscd}
\usepackage{epsfig}
\usepackage{verbatim}
\usepackage{latexsym,amstext,epsfig}

\usepackage[all, knot]{xy}
\xyoption{arc}

\newsymbol\pp 1275
\newcommand{\Hom}{\operatorname{Hom}}
\newcommand{\End}{\operatorname{End}}
\newcommand{\Ext}{\operatorname{Ext}}
\newcommand{\ext}{\operatorname{ext}}
\newcommand{\Rep}{\operatorname{Rep}}

\newcommand{\SI}{\operatorname{SI}}
\newcommand{\SL}{\operatorname{SL}}
\newcommand{\GL}{\operatorname{GL}}

\newcommand{\ZZ}{\mathbb Z}

\newcommand{\RR}{\mathbb R}

\newcommand{\coker}{\operatorname{coker}}

\newcommand{\supp}{\operatorname{supp}}

\newtheorem{theorem}{Theorem}[section]
\newtheorem{proposition}[theorem]{Proposition}

\newtheorem{lemma}[theorem]{Lemma}

\theoremstyle{definition}
\newtheorem{definition}[theorem]{Definition}
\newtheorem{remark}[theorem]{Remark}

\newtheorem{example}[theorem]{Example}

\title[]{Orbit semigroups and the representation type of quivers}

\author{Calin Chindris}
\address{University of Iowa, Department of Mathematics, Iowa City, IA 52242, USA}
\email[Calin Chindris]{calin-chindris@uiowa.edu}

\date{25 August, 2007; Revised: \today}
\begin{document}
\bibliographystyle{plain}
\begin{abstract}
We show that a finite, connected quiver $Q$ without oriented
cycles is a Dynkin or Euclidean quiver if and only if all orbit
semigroups of representations of $Q$ are saturated.
\end{abstract}
\maketitle

\section{Introduction}
The representation type of a quiver reflects the complexity of its
indecomposable representations. There are three distinct classes:
\emph{finite type}, \emph{tame}, and \emph{wild} quivers. A quiver
is said to be of finite type if there are only finitely many
indecomposable representations. We say that a quiver is tame if it
is not of finite representation type, and in each dimension all
but finitely many indecomposable representations come in a finite
number of $1$-parameter families. Finally, we call a quiver wild
if its representation theory is at least as complicated as that of
a free algebra in two (non-commuting) variables. For precise
definitions, we refer to \cite[Ch. 4]{Benson}.

Gabriel's classical result \cite{Ga} identifies the connected
quivers of finite type as being those whose underlying graphs are
the Dynkin diagrams of types $\mathbb{A}$, $\mathbb{D}$, or
$\mathbb{E}$. Later on, Nazarova \cite{Naz} and Donovan-Freislich
\cite{DF} found the tame, connected quivers. Their underlying
graphs are the Euclidean diagrams of types
$\widetilde{\mathbb{A}}$, $\widetilde{\mathbb{D}}$, or
$\widetilde{\mathbb{E}}$. The remaining connected quivers are the wild ones.

It is an important and interesting task to find geometric
characterizations of the representation type of a quiver (or more
generally, of finite dimensional algebras). In \cite[Theorem
1]{SW1}, Skowro{\'n}ski and Weyman showed that a finite, connected
quiver is a Dynkin or Euclidean quiver if and only if the various
algebras of semi-invariants are always complete intersections. In
this paper, we provide a different characterization of the
representation type in terms of saturated orbit semigroups.

Let $Q$ be a quiver and $\beta$ a dimension vector. Following
\cite[Definition 2.1]{BeHa}, we define the orbit semigroup of a
representation $W \in \Rep(Q,\beta)$ to be
$$
S(W)_Q=\{\sigma \in \ZZ^{Q_0} \mid \exists f \in
\SI(Q,\beta)_{\sigma} \text{~such that~} f(W)\neq 0\}.
$$
The cones generated by the semigroups $S(W)_Q$ play a fundamental
role in the construction of the GIT-fans for quivers (see
\cite{CC5} and the reference therein). Furthermore, Derksen-Weyman
saturation theorem \cite{DW1} for semi-invariants tells us that
$S(W)_Q$ are saturated for \emph{generic} representations. However,
there are quiver representations whose orbit semigroups are not
saturated. We refer to Section \ref{intro-sec} for background
material on quiver invariant theory. Throughout this paper, we work over an algebraically closed field
$k$ of characteristic zero.

Now, we are ready to state our main result:

\begin{theorem}\label{main-thm} Let $Q$ be a finite, connected quiver without oriented cycles.
The following are equivalent:
\begin{enumerate}
\renewcommand{\theenumi}{\arabic{enumi}}
\item $Q$ is a Dynkin or Euclidean quiver;

\item for every dimension vector $\beta$, the semigroup $S(W)_Q$ is
saturated for every $W \in \Rep(Q,\beta)$.
\end{enumerate}
\end{theorem}

%We should point out that our theorem remains true for other
%classes of algebras (see Section \ref{beyond-sec}).

To show that orbit semigroups for Dynkin or Euclidean quivers are
saturated, we use Derksen-Weyman spanning theorem and a theorem of
Schofield on Kac's canonical decomposition for dimension vectors.
This allows us to give a short, conceptual proof avoiding a
case-by-case analysis. To deal with wild quivers, we use
reflection functors, shrinking methods and exceptional sequences
to reduce the list of wild quivers to the generalized Kronecker
quiver with three arrows.

The layout of this paper is as follows. Background material on
quiver invariant theory is reviewed in Section \ref{intro-sec}. In
particular, we recall Derksen-Weyman spanning and saturation
theorems. In Section \ref{Kac-sec}, we first recall Schofield's
theorem on canonical decompositions and then we prove the
necessary part of our theorem. Reflection functors, the shrinking
method, and exceptional sequences are reviewed in Section
\ref{reduction-sec} where we show that these reduction methods
behave nicely with respect to saturated orbit semigroups. We
complete the proof of Theorem \ref{main-thm} in Section
\ref{wild-sec} by showing that for every wild quiver without
oriented cycles there is a representation whose orbit semigroup is
not saturated. %In Section \ref{beyond-sec}, we give more examples
%of algebras for which our theorem remains true.
The last section discusses the thin sincere case. In this case, we show that the
orbit semigroups are saturated.

\subsection*{Acknowledgement} I am grateful to Harm Derksen and
Jerzy Weyman for helpful discussions on the subject, especially
regarding some of the examples from  Section \ref{wild-sec}. I
would also like to thank Ben Howard for stimulating discussions on
his results from \cite{BenH}. Many thanks are also due to the referee for a very thorough report which helped improved the paper.

\section{Recollection from quiver invariant theory}
\label{intro-sec}

Let $Q=(Q_0,Q_1,t,h)$ be a finite quiver, where $Q_0$ is the set
of vertices, $Q_1$ is the set of arrows and $t,h:Q_1 \to Q_0$
assign to each arrow $a \in Q_1$ its tail \emph{ta} and head
\emph{ha}, respectively.

In this paper, we work over an algebraically closed field
$k$ of characteristic zero. A representation $V$ of $Q$ over $k$
is a collection $(V(x), V(a))_
{x\in Q_0, a \in Q_1}$ of finite dimensional $k$-vector spaces $V(x)$, $x\in Q_0$, and $k$-linear maps $V(a)\in \Hom_{k}(V(ta), V(ha))$, $a \in Q_1$. If $V$ is a representation of $Q$, we define its dimension vector
$\underline d_V$ by $\underline d_V(x)=\dim_{k} V(x)$ for every
$x\in Q_0$. Thus the dimension vectors of representations of $Q$
lie in $\Gamma=\ZZ^{Q_0}$, the set of all integer-valued functions
on $Q_0$.

Given two representations $V$ and $W$ of $Q$, we define a morphism
$\varphi:V \rightarrow W$ to be a collection of $k$-linear maps
$(\varphi(x))_{x \in Q_0}$ with $\varphi(x)\in \Hom_k(V(x), W(x))$, $x \in Q_0$, and
such that
$\varphi(ha)V(a)=W(a)\varphi(ta)$ for every arrow $a \in Q_1$. We denote by $\Hom_Q(V,W)$ the
$k$-vector space of all morphisms from $V$ to $W$. Let $W$ and $V$
be two representations of $Q.$ We say that $V$ is a
subrepresentation of $W$ if $V(x)$ is a subspace of $W(x)$ for all
vertices $x \in Q_0$ and $V(a)$ is the restriction of $W(a)$ to
$V(ta)$ for all arrows $a \in Q_1$. In this way, we obtain the
abelian category $\Rep(Q)$ of all quiver representations of $Q$.

A representation $W$ is said to be a \emph{Schur representation}
if $\End_{Q}(W) \cong k$. The dimension vector of a Schur
representation is called a \emph{Schur root}.

\textbf{From now on, we assume that our quivers are without
oriented cycles}. For two quiver representations $V$ and $W$,
consider Ringel's canonical exact sequence \cite{R} :
\begin{equation}\label{can-exact-seq}
0 \rightarrow \Hom_Q(V,W) \rightarrow \bigoplus_{x\in
Q_0}\Hom_{k}(V(x),W(x)){\buildrel
d^V_W\over\longrightarrow}\bigoplus_{a\in
Q_1}\Hom_{k}(V(ta),W(ha)),%\rightarrow \Ext^1(V,W)\rightarrow 0
\end{equation}
where $d^V_W((\varphi(x))_{x\in
Q_0})=(\varphi(ha)V(a)-W(a)\varphi(ta))_{a\in Q_1}$ and
$\Ext^1_{Q}(V,W)=\coker(d^V_W).$

If $\alpha,\beta$ are two elements of $\Gamma$, we define the
Euler inner product by
\begin{equation}\label{Euler-prod}
\langle\alpha,\beta \rangle_{Q} = \sum_{x \in Q_0}
\alpha(x)\beta(x)-\sum_{a \in Q_1} \alpha(ta)\beta(ha).
\end{equation}
(When no confusion arises, we drop the subscript $Q$.) It follows
from (\ref{can-exact-seq}) and (\ref{Euler-prod}) that
$$\langle\underline d_V, \underline d_W\rangle=
\dim_k\Hom_{Q}(V,W)-\dim_k\Ext^1_{Q}(V,W).$$

\subsection{Semi-invariants of quivers} For every vertex $x$, we denote by
$e_x$ the simple dimension vector corresponding to $x$, i.e.,
$e_x(y)=\delta_{x,y}, \forall y\in Q_0$, where $\delta_{x,y}$ is
the Kronecker symbol.

Let $\beta$ be a dimension vector of $Q$. The representation space
of $\beta-$dimensional representations of $Q$ is
$$\Rep(Q,\beta)=\bigoplus_{a\in Q_1}\Hom_{k}(k^{\beta(ta)}, k^{\beta(ha)}).$$
If $\GL(\beta)=\prod_{x\in Q_0}\GL(\beta(x))$ then $\GL(\beta)$
acts algebraically on $\Rep(Q,\beta)$ by simultaneous conjugation,
i.e., for $g=(g(x))_{x\in Q_0}\in \GL(\beta)$ and $W=(W(a))_{a
\in Q_1} \in \Rep(Q,\beta),$ we define $g \cdot W$ by
$$(g\cdot W)(a)=g(ha)W(a)g(ta)^{-1}\ \text{for each}\ a \in Q_1.$$ Hence, $\Rep(Q,\beta)$ is a rational representation of the
linearly reductive group $\GL(\beta)$ and the $\GL(\beta)-$orbits
in $\Rep(Q,\beta)$ are in one-to-one correspondence with the
isomorphism classes of $\beta-$dimensional representations of $Q.$
As $Q$ is a quiver without oriented cycles, one can show that
there is only one closed $\GL(\beta)-$orbit in $\Rep(Q,\beta)$ and
hence the invariant ring $\text{I}(Q,\beta)= k
[\Rep(Q,\beta)]^{\GL(\beta)}$ is exactly the base field $k$.

Now, consider the subgroup $\SL(\beta) \subseteq \GL(\beta)$
defined by
$$
\SL(\beta)=\prod_{x \in Q_0}\SL(\beta(x)).
$$

Although there are only constant $\GL(\beta)-$invariant polynomial
functions on $\Rep(Q,\beta)$, the action of $\SL(\beta)$ on
$\Rep(Q,\beta)$ provides us with a highly non-trivial ring of
semi-invariants. Note that any $\sigma \in \ZZ^{Q_0}$ defines a
rational character of $\GL(\beta)$ by
$$(g(x))_{x \in Q_0} \in \GL(\beta) \mapsto \prod_{x \in
Q_0}(\det g(x))^{\sigma(x)}.$$ In this way, we can identify
$\Gamma=\ZZ ^{Q_0}$ with the group $X^\star(\GL(\beta))$ of
rational characters of $\GL(\beta),$ assuming that $\beta$ is a
sincere dimension vector (i.e., $\beta(x)>0$ for all vertices $x
\in Q_0$). In general, we have only the epimorphism $\Gamma \to X^*(\GL(\beta))$, but we usually do not distinguish between $\sigma$ and its image in $X^*(\GL(\beta))$. We also refer to the rational characters of $\GL(\beta)$ as (integral) weights.

Let $\SI(Q,\beta)= k[\Rep(Q,\beta)]^{\SL(\beta)}$ be the ring of
semi-invariants. As $\SL(\beta)$ is the commutator subgroup of
$\GL(\beta)$ and $\GL(\beta)$ is linearly reductive, we have
$$\SI(Q,\beta)=\bigoplus_{\sigma
\in X^\star(\GL(\beta))}\SI(Q,\beta)_{\sigma},
$$
where $$\SI(Q,\beta)_{\sigma}=\lbrace f \in k[\Rep(Q,\beta)] \mid
g\cdot f= \sigma(g)f \text{~for all~}g \in \GL(\beta)\rbrace$$ is
the space of semi-invariants of weight $\sigma$.

If $\alpha \in \ZZ^{Q_0}$, we define the weight $\sigma=\langle
\alpha,\cdot \rangle$ by
$$\sigma(x)=\langle \alpha,e_x \rangle, \forall x\in
Q_0.$$ Conversely, it is easy to see that for any weight $\sigma
\in \ZZ^{Q_0}$ there is a unique $\alpha \in \ZZ^{Q_0}$ (not
necessarily a dimension vector) such that $\sigma=\langle
\alpha,\cdot \rangle$. Similarly, one can define $\mu = \langle
\cdot,\alpha \rangle$.

In \cite{S2}, Schofield constructed semi-invariants of quivers with
remarkable properties. Let $\alpha,\beta$ be two dimension vectors
such that $\langle \alpha,\beta \rangle=0$. Following \cite{S2},
we define
$$
c:\Rep(Q,\alpha)\times \Rep(Q,\beta)\rightarrow k
$$ by $c(V,W)=\det(d^V_W)$. Note that $d^V_W$ from (\ref{can-exact-seq})
is indeed a square matrix since $\langle \alpha,\beta \rangle=0$.
Fix $(V,W)\in \Rep(Q,\alpha) \times \Rep(Q,\beta)$. Then it is
easy to see that $c^V=c(V,\cdot):\Rep(Q,\beta)\rightarrow k$ is a
semi-invariant of weight $\langle \alpha, \cdot \rangle$ and
$c_W=c(\cdot,W):\Rep(Q,\alpha)\rightarrow k$ is a semi-invariant
of weight $-\langle \cdot, \beta \rangle$.

\begin{remark} We should point out that if $V$ is an $\alpha$-dimensional
representation not necessarily in $\Rep(Q,\alpha)$, the semi-invariant $c^{V}$ is well-defined on $\Rep(Q,\beta)$ up to a non-zero scalar. %Also, if $V$
%and $V'$ are two $\alpha$-dimensional representations and $V\cong
%V'$ then $c^{V}$ and $c^{V'}$ are equal up to a non-zero scalar.
\end{remark}

Given two representations $V$ and $W$, we say that $V$ is
orthogonal to $W$, and write $V \perp W$, if
$\Ext^1_{Q}(V,W)=\Hom_Q(V,W)=0$.

\begin{remark} \label{non-zero-semi-rmk}
Using the exact sequence (\ref{can-exact-seq}) for two representations $V$ and $W$, we deduce that
$c^V(W)\neq 0$ if and only if $V \perp W$. %Of course, if either $V$ or $W$ is the zero representation then $d^V_W$ is the unique map between two zero dimensional vector spaces, in which case, we set $c^V(W)$ to be just one.
\end{remark}

\subsection{The Spanning and Saturation Theorems}

A very important property of the Schofield semi-invariants is that
each weight space of semi-invariants is spanned by such
semi-invariants. This is a fundamental result proved by Derksen
and Weyman \cite{DW1} (see also \cite{SVB}).

\begin{theorem}[The Spanning Theorem] \label{spanning-thm}\cite{DW1}
Let $\beta$ be a sincere dimension vector of $Q$ and let $\sigma=\langle \alpha, \cdot \rangle$ be a weight with $\alpha \in \ZZ^{Q_0}$. If the weight space of semi-invariants $\SI(Q,\beta)_{\sigma}$ is non-zero then $\alpha$ is a dimension vector, and moreover, $\SI(Q,\beta)_{\sigma}$ is spanned by the semi-invariants $c^V$ with $V \in \Rep(Q,\alpha)$.
\end{theorem}

Recall that for a given representation $W \in \Rep(Q,\beta)$, its orbit semigroup is
$$
S(W)_Q=\{\sigma \in \ZZ^{Q_0} \mid \exists f \in
\SI(Q,\beta)_{\sigma} \text{~such that~} f(W)\neq 0\}.
$$
(When no confusion arises, we drop the subscript $Q$.)

Now, let $Q'$ be the full subquiver of $Q$ whose vertex set is $\supp(W)=\{x \in Q_0 \mid
W(x)\neq \{0\} \}$ and denote by $W'$ the restriction of $W$ to $Q'$. If $\sigma \in \ZZ^{Q_0}$ is a weight, we
denote by $\sigma'$ its restriction to $Q'$. Also, any
representation $V'$ of $Q'$ is viewed as a representation of $Q$
in a natural way. Note that $S(W)_Q$ is just the inverse image in $\ZZ^{Q_0}$ of $S(W')_{Q'}$.

From Remark \ref{non-zero-semi-rmk} and Theorem \ref{spanning-thm}, we deduce:

\begin{proposition}\label{orbit-semi-def-prop} Let $W \in \Rep(Q,\beta)$ be a representation and
$\sigma=\langle \alpha, \cdot \rangle$ a weight with $\alpha \in
\ZZ_{\geq 0}^{Q_0}$. Then,

$$\sigma \in S(W) \Longleftrightarrow \exists V \in \Rep(Q,\alpha) \text{~such that~} V \perp W.$$

Consequently,

\begin{equation} \label{def-orbitsemi}
S(W)=\{\sigma \in \ZZ^{Q_0} \mid \sigma'=\langle \alpha', \cdot
\rangle_{Q'} \text{~with~} \alpha' \in \ZZ^{Q'_0}_{\geq 0} \text{~and~} \exists V' \in \Rep(Q',\alpha')
\text{~such that~} V' \perp W \}.
\end{equation}
\end{proposition}

At this point, we can use $(\ref{def-orbitsemi})$ to define orbit
semigroups for arbitrary representations. In this way, the orbit
semigroup $S(W)$ is independent of the choice of the isomorphism
class of $W$. We should mention that Proposition
\ref{orbit-semi-def-prop} together with Theorem
\ref{canonical-decomp-thm} below plays a crucial role in our
study.

Using the Spanning Theorem and Schofield's theory of general
representations \cite{S1}, Derksen and Weyman proved the following
remarkable result:

\begin{theorem}[The Saturation Theorem] \cite{DW1} \label{DW-sat}
Let $Q$ be a quiver and $\beta$ a dimension vector. Consider the
semigroup of integral effective weights:
$$
\Sigma(Q,\beta)=\{\sigma \in \ZZ^{Q_0} \mid \SI(Q,\beta)_{\sigma}
\neq 0 \}.
$$
Then, $\Sigma(Q,\beta)$ is saturated, i.e., for any positive
integer $n$ and $\sigma \in \ZZ^{Q_0}$,
$$
n\sigma \in \Sigma(Q,\beta) \Longleftrightarrow \sigma \in
\Sigma(Q,\beta).
$$

%Then
%$$m\alpha \circ \beta \neq 0 \Longleftrightarrow \alpha \circ \beta \neq 0 \Longleftrightarrow \alpha \circ n\beta \neq 0,$$
%for any integers $m,n \geq 1$.
\end{theorem}

\begin{remark} Let $\{ f_1, \dots, f_m \}$ be a generating system
of semi-invariants for $\SI(Q,\beta)$. Then, every representation
$W \in \Rep(Q,\beta)$ with $f_i(W) \neq 0, \forall 1 \leq i \leq
m$ has the property that $S(W)=\Sigma(Q,\beta)$. This shows that
orbit semigroups are saturated for generic representations.
\end{remark}

\begin{remark}
It is worth pointing out that for star quivers, Theorem
\ref{DW-sat} implies the Saturation Conjecture for
Littlewood-Richardson coefficients (for more details, see
\cite{DW1} and \cite{DW2}).
\end{remark}

\section{Kac's canonical decomposition}\label{Kac-sec}
%One of the main tools that we use in this paper is
%Derksen-Weyman's notion of $\sigma$-stable decomposition of a
%dimension vector. Based on Theorem \ref{King-criterion}, we can
%construct the full subcategory $\Rep(Q)^{ss}_{\sigma}$ of
%$\Rep(Q)$ consisting of all $\sigma$-semi-stable representations
%of $Q$. It turns out that $\Rep(Q)^{ss}_{\sigma}$ is an abelian
%category. It follows from Theorem \ref{King-criterion} that the
%simple objects of $\Rep(Q)^{ss}_{\sigma}$ are precisely the
%$\sigma$-stable representations. Furthermore, every object of
%$\Rep(Q)^{ss}_{\sigma}$ has a Jordan-H$\ddot{o}$lder filtration
%whose factors are $\sigma$-stable representations.

%Now, let $\beta$ be a $\sigma$-semi-stable dimension vector.
%Following \cite{DW2}, we say that $$\beta=\beta_1 \pp \beta_2 \pp
%\ldots \pp \beta_l$$ is the \emph{$\sigma$-stable decomposition}
%of $\beta$ if a general representation in $\Rep(Q,\beta)$ has a
%Jordan-H$\ddot{o}$lder filtration in $\Rep(Q)^{ss}_{\sigma}$ with
%factors of dimension $\beta_1, \ldots ,\beta_l$ in some order (for
%more details, we refer to \cite{DW2}). We write $c\cdot\beta$
%instead of $\underbrace{\beta \pp \beta \pp \ldots \pp
%\beta}_{c}$.

One of the main tools that we use in this paper is Kac's canonical
decomposition of dimension vectors. Let $Q$ be a quiver and
$\alpha$ a dimension vector. Following \cite{Kac2}, we say that
$$
\alpha=\alpha_1 \oplus \dots \oplus \alpha_l
$$
is the canonical decomposition of $\alpha$ if there is a non-empty
open subset $\mathcal{U} \subseteq \Rep(Q,\alpha)$ such that every
$V \in \mathcal{U}$ decomposes as a direct sum of indecomposables
of dimension vectors $\alpha_1, \dots, \alpha_l$. It was proved by Kac
that the dimension vectors occurring in the canonical
decomposition of $\alpha$ must be Schur roots.

Recall that a root of $Q$ is just the dimension vector of an
indecomposable representation of $Q$. We say that a root $\alpha$
is \emph{real} if $\langle \alpha, \alpha \rangle=1$. If $\langle
\alpha, \alpha \rangle=0$, $\alpha$ is said to be an
\emph{isotropic} root. Finally, we say that $\alpha$ is a
\emph{non-isotropic imaginary} root if $\langle \alpha, \alpha
\rangle<0$.

It is important to know how to obtain the canonical decomposition
of a multiple of $\alpha$ from that of $\alpha$. Schofield's
theorem gives an answer to this question:

\begin{theorem} \cite[Theorem 3.8]{S1} \label{canonical-decomp-thm}
Let $\alpha=\alpha_1 \oplus \dots \oplus \alpha_l$ be the
canonical decomposition of $\alpha$ and let $m \geq 1$. Then, the
canonical decomposition of $m\alpha$ is
$$m\alpha=[m\alpha_1]\oplus
\dots \oplus [m\alpha_l],$$ where
$$[m\alpha_i]=\begin{cases}
\alpha_i^{\oplus m}:=\alpha_i \oplus \dots \oplus \alpha_i &
\text{if $\alpha_i$ is a real or
isotropic Schur root},\\
m\alpha_i & \text{if $\alpha_i$ is a non-isotropic imaginary Schur
root.}
\end{cases}
$$
\end{theorem}

Now, we are ready to prove:

\begin{proposition} \label{tamequivers-prop} Let $Q$ be a Dynkin or Euclidean quiver. Then the orbit semigroup
$S(W)$ is saturated for every representation $W$.
\end{proposition}

\begin{proof}
%Assume that $Q$ is a Dynkin quiver. In this case, for every
%dimension vector $\beta$, $\GL(\beta)$ acts with a dense orbit on
%$\Rep(Q,\beta)$. Hence, the dimensions of the weight spaces of
%semi-invariants are at most one. From this observation and Theorem
%\ref{DW-sat}, it follows that the orbit semigroups are saturated.

Assume that $Q$ is a Dynkin or Euclidean quiver and let $W \in
\Rep(Q,\beta)$. We can clearly assume that $\beta$ is a sincere
dimension vector. Let $\sigma=\langle \alpha, \cdot \rangle \in
\ZZ^{Q_0}$ be a weight with $\alpha \in \ZZ^{Q_0}$. Assume
$n\sigma \in S(W)$ for some $n \geq 1$. From Theorem \ref{spanning-thm}, it follows that $\alpha$ is a dimension
vector. %We know that $n \alpha \circ \beta \neq 0$ and this is
%equivalent to $\alpha \circ \beta \neq 0$ by Theorem \ref{DW-sat}.
Now, consider the canonical decomposition of $\alpha$:
$$
\alpha=\alpha_1 \oplus \dots \oplus \alpha_l.
$$
Since $Q$ is a Dynkin or Euclidean quiver, the Schur roots $\alpha_i$ are
either real or isotropic. Using Theorem
\ref{canonical-decomp-thm}, we obtain that the canonical
decomposition of $n\alpha$ is:
$$
n\alpha=\alpha_1^{\oplus n} \oplus \dots \oplus \alpha_l^{\oplus
n}.
$$

It is well-known that the functions $\dim_k \Ext_Q^1(\cdot,W)$ and
$\dim_k \Hom_Q(\cdot,W)$ are upper semi-continuous. This fact
combined with Proposition \ref{orbit-semi-def-prop} allows us to
find a representation $V \in \Rep(Q,n\alpha)$ such that $V \perp W$ and $V$ has a direct summand of the form $\bigoplus_{i=1}^l V_i$ with $V_i$ an
indecomposable representation of dimension vector
$\alpha_i$, $1 \leq i \leq l$. Set $\widetilde{V}=\bigoplus_{i=1}^l V_i$. Then, $\widetilde{V}$ is an $\alpha$-dimensional representation with $\widetilde{V} \perp W$ and hence $\sigma \in S(W)$ by Proposition \ref{orbit-semi-def-prop}.
%By Lemma
%\ref{direct-sums-lemma}, we know that $\prod_{i=1}^l c^{V_{i1}}$
%divides $c^{V}$ and so
%$$
%\prod_{i=1}^l c^{V_{i1}}(W) \neq 0.
%$$
%Setting $\widetilde{V}=\bigoplus_{i=1}^l V_{i1}$, we get
%$c^{\widetilde{V}}(W)=\prod_{i=1}^l c^{V_{i1}}(W) \neq 0$. Hence,
%$\sigma \in S(W)$ and we are done.
\end{proof}

\section{Reflection functors, the shrinking method, and exceptional sequences}\label{reduction-sec}

In this section, we describe three reduction methods that behave
nicely with respect to orbit semigroups. This will be particularly
useful when dealing with wild quivers.

A quiver $Q$ is said to satisfy \emph{property $(S)$} provided the
following is true: For every representation $W$ of $Q$, the orbit
semigroup $S(W)$ is saturated. It is clear that if $Q$ has
property $(S)$ then any (not necessarily full) subquiver has it.

\subsection{Reflection functors} Let $Q$ be a quiver and $x \in Q_0$ a vertex. Define
$s_{x}(Q)$ to be the quiver obtained from $Q$ by reversing all
arrows incident to $x$. The reflection transformation $s_x :
\ZZ^{Q_0} \to \ZZ^{Q_0}$ at vertex $x$ is defined by
$$
s_x(\alpha)(y)=\begin{cases}
\alpha(y) & \text{if~} y \neq x,\\
\sum_{ha=x}\alpha(ta)+\sum_{ta=x}\alpha(ha)-\alpha(x) & \text{if~}
y=x.
\end{cases}
$$
Note that $s_x$ is the reflection in the plane orthogonal to the
simple root $e_x$.

Now, let us assume that $x$ is a sink. The
Bernstein-Gelfand-Ponomarev reflection functor at $x$ is defined
as follows:
$$
\begin{aligned}
C^{+}_x : \Rep(Q) & \to \Rep(s_x(Q))\\
V&\mapsto W=C^{+}_x(V),
\end{aligned}
$$
where $W(y)=V(y)$ for $y \neq x$ and $W(x)=\ker(\bigoplus_{ha=x}
V(ta) \to V(x))$.

If $x$ is a source, we define $C^{-}_{x}$ by
$$
\begin{aligned}
C^{-}_x : \Rep(Q) & \to \Rep(s_x(Q))\\
V&\mapsto W=C^{-}_x(V),
\end{aligned}
$$
where $W(y)=V(y)$ for $y \neq x$ and $W(x)=\coker(V(x) \to
\bigoplus_{ta=x} V(ha))$.

The following theorem, proved by Bernstein-Gelfand-Ponomarev, is
one of the fundamental results about reflection functors.

\begin{theorem}\cite{BGP}\label{BGP} Let $Q$ be a quiver and $x \in Q_0$ a sink.
\begin{enumerate}
\renewcommand{\theenumi}{\arabic{enumi}}

\item If $V=S_x$ then $C^{+}_x(V)=0$.

\item If $V\neq S_x$ is indecomposable then $C^{+}_x(V)$ is
indecomposable, too. Furthermore, $C^{-}_xC^{+}_x(V)=V$ and
$$
\underline{d}_{C^{+}_x(V)}=s_x(\underline{d}_{V}).
$$
%$$
%\dim_k C^{+}_x(V)(y)=
%\begin{cases}

%\dim_k V(y) & \text{~if~} y \neq x, \\
%\sum_{ha=x} \dim_k V(ta)-\dim_k V(x) & \text{~if~} y=x.
%\end{cases}
%$$
\end{enumerate}

The analogous result for $C^{-}_x$ with $x$ a source is also true.
\end{theorem}

Using the First Fundamental Theorem for special linear groups (see
for example \cite{DK}), Kac \cite[Sections 2 and 3]{Kac} showed
that reflection functors behave nicely with respect to
semi-invariants of quivers. Other geometric properties that are
preserved by reflection functors can be found in \cite[Section
6]{CW}. We are going to show that property $(S)$ defined above is
also preserved by reflection functors. The following proposition
will be very useful for us:

\begin{proposition}\cite{APR}\label{semi-reflect-prop} Let $V$ and $W$ be
two representations of $Q$. Assume that $x$ is a sink and $S_x$ is
not a direct summand of $V$ or $W$. Then:
$$
\dim_k \Hom_{Q}(V,W)=\dim_k \Hom_{s_x(Q)}(C^{+}_x(V),C^{+}_x(W)),
$$
and
$$
\dim_k \Ext^1_{Q}(V,W)=\dim_k \Ext^1_{s_x(Q)}(C^{+}_x(V),C^{+}_x(W)).
$$

Consequently, one has $$V \perp W \Longleftrightarrow C^{+}_x(V)
\perp C^{+}_x(W).$$ The same is true when $x$ is a source and
$C^{+}_x$ is replaced by $C^{-}_x$.
\end{proposition}

Next, we show that when checking property $(S)$, we can actually
avoid those representations that have direct summands isomorphic
to simple representations. We write $S_x \nmid W$ if $W$ does not
have a direct summand isomorphic to $S_x$. A representation $W$
whose dimension vector is sincere is called a \emph{sincere
representation}.

\begin{lemma}\label{semigr-directsum-lemma}
\begin{enumerate}
\renewcommand{\theenumi}{\arabic{enumi}}
\item Let $W=W_1 \oplus W_2$ be a representation with $S(W_1)$ and
$S(W_2)$ saturated. Then, $S(W)$ is saturated, too.

\item Let $x$ be a vertex of $Q$. Then, $Q$ has property $(S)$ if
and only if $S(W)$ is saturated for every representation $W$ such
that $S_x \nmid W$.

\end{enumerate}
\end{lemma}

\begin{proof}$(1)$ We can clearly assume that $W$ is a sincere representation. From Proposition \ref{orbit-semi-def-prop}, it follows that
$$
S(W)\subseteq S(W_1) \cap S(W_2).
$$

Now, let $\sigma \in \ZZ^{Q_0}$ be so that $n\sigma \in S(W)$ for
some positive integer $n$. Then, $\sigma=\langle \alpha, \cdot
\rangle$ for a unique dimension vector $\alpha$, and since
$S(W_i)$ are assumed to be saturated, we obtain that $\sigma \in
S(W_1) \cap S(W_2)$. This is equivalent to the existence of $V_i
\in \Rep(Q,\alpha)$ such that
$$
\Ext^1_Q(V_i,W_i)=\Hom_Q(V_i,W_i)=0,
$$
for $i \in \{1,2\}$. Using the upper semi-continuity of the
functions $\dim_k \Ext_Q^1(\cdot,W_i)$ and $\dim_k \Hom_Q(\cdot,W_i)$,
$i \in \{1,2\}$, we can find a representation $V\in
\Rep(Q,\alpha)$ such that $\Ext_Q^1(V,W_i)=\Hom_Q(V,W_i)=0,
\forall i \in \{1,2\}$. Applying Proposition
\ref{orbit-semi-def-prop} again, we get $\sigma \in S(W)$.

$(2)$ The proof follows from part $(1)$ and the fact that the
orbit semigroup of $S_x$ is clearly saturated.
\end{proof}

For our purposes, we actually need to strengthen Lemma
\ref{semigr-directsum-lemma}(2):

\begin{lemma} \label{sat-orbit-semigr-lemma} Let $W \in \Rep(Q,\beta)$ and $x \in Q_0$. Consider the
set
$$
S_x(W)=\{\sigma= \langle \alpha, \cdot \rangle \in \ZZ^{Q_0} \mid
\exists V \in \Rep(Q,\alpha) \text{~such that~} V \perp W
\text{~and~} S_x \nmid V \}.
$$
If $S_x(W)$ is saturated then so is $S(W)$.
\end{lemma}

\begin{proof}
We distinguish two cases:

\emph{Case 1:} $W$ is a sincere representation. Let
$\sigma=\langle \alpha, \cdot \rangle \in \ZZ^{Q_0}$ be a weight and assume that there exists an integer $n>1$ such that $n\sigma \in
S(W)$.

We know that $\alpha$ has to be a
dimension vector by Theorem \ref{spanning-thm}. Consider the canonical decomposition of
$\alpha$:
$$
\alpha=\alpha_1 \oplus \dots \oplus \alpha_l.
$$
Without loss of generality, let us assume that $\alpha_1, \dots,
\alpha_m$ are the non-isotropic, imaginary Schur roots in the
decomposition above (of course, we allow $m$ to be zero). From
Theorem \ref{canonical-decomp-thm}, we obtain that
$$
n\alpha=n\alpha_1 \oplus \dots \oplus n\alpha_m\oplus
\alpha_{m+1}^{\oplus n} \oplus \dots \oplus \alpha_l^{\oplus n}
$$
is the canonical decomposition of $n\alpha$.

Now, choose an $n\alpha$-dimensional representation $V$ such that
$V \perp W$ and $$V=\bigoplus_{i=1}^m V'_i \oplus \bigoplus_{\buildrel {m+1 \leq j \leq l} \over {1 \leq k \leq n}
}V_{jk}$$ where the $V'_i$ are indecomposables of dimension vectors $n\alpha_i$, $1 \leq i \leq m$, and the $V_{jk}$ are indecomposables of dimension vectors $\alpha_j$, $1 \leq k \leq n, m+1 \leq j\leq l$.

Note that $V'_i \perp W$ and $S_x \nmid V'_i$ for every $ 1 \leq i
\leq m$. Since $S_x(W)$ is assumed to be saturated, it follows that
$\langle \alpha_i, \cdot \rangle \in S_x(W)$. So, we can choose
$\alpha_i$-dimensional representations $V_i$ such that
$V_i \perp W$ for all $1\leq i \leq m$.

Set $U=\bigoplus_{i=1}^{m} V_i \oplus \bigoplus_{j=m+1}^{l} V_{j1}$. Then,
$U$ is an $\alpha$-dimensional representation with $U \perp W$ and this finishes the proof in the sincere case.

\emph{Case 2:} $W$ is not necessarily a sincere representation.
Let $Q'$ be the full subquiver of $Q$ whose vertex set is
$\supp(W)=\{x \in Q_0 \mid W(x)\neq \{0\} \}$. Denote by $W'$ the
restriction of $W$ to $Q'$. Then, $W'$ is a sincere representation
of $Q'$ and we clearly have that $S(W)_{Q}$ is saturated if and
only if $S(W')_{Q'}$ is saturated.

Now, consider the following linear transformation:
$$
\begin{aligned}
I: \ZZ^{Q'_0}& \to \ZZ^{Q_0}\\
\sigma'=\langle \alpha', \cdot \rangle_{Q'}& \to \langle \alpha',
\cdot \rangle_{Q},
\end{aligned}
$$
where any $\alpha' \in \ZZ^{Q'_0}$ is viewed as an element of
$\ZZ^{Q_0}$ in a natural way. Given a weight $\sigma' \in
\ZZ^{Q'_0}$, we clearly have $\sigma' \in S_{x}(W')_{Q'}$ if and
only if $I(\sigma') \in S_{x}(W)_Q$. Hence, $S_x(W')_{Q'}$ must be
saturated. But now, this implies that $S(W)_{Q}$ is saturated, as
well.

\end{proof}

Now, we are ready to prove:

\begin{proposition}\label{tool-reflection} Let $Q$ be a quiver without oriented cycles and let $x$ be either a source or a
sink. Then, $Q$ satisfies property $(S)$ if and only if so does
$s_x(Q)$.
\end{proposition}

\begin{proof} It is enough to show that $Q$ satisfies property $(S)$ assuming that $x$ is a source and $\widetilde{Q}=s_x(Q)$ satisfies
$(S)$. Let $W \neq 0$ be a representation of $Q$. By Lemma
\ref{semigr-directsum-lemma}, we can assume that $S_x \nmid W$ and
let us denote $C^{-}_x(W)$ by $\widetilde{W}$. Then, it follows from Theorem
\ref{BGP} that $S_x \nmid \widetilde{W}$ and $C^{+}_x(\widetilde{W})=W$.

Now, let $\alpha$ be a dimension vector of $Q$. Using Theorem
\ref{BGP} again and then Proposition \ref{semi-reflect-prop}, we
deduce that $\langle \alpha, \cdot \rangle_{Q} \in S_x(W)_{Q}$ if
and only if $\langle s_x(\alpha), \cdot \rangle_{\widetilde{Q}} \in
S_x(\widetilde{W})_{\widetilde{Q}}$. But this latter set is saturated by assumption and
hence $S_x(W)$ must be saturated. The proof follows now from Lemma
\ref{sat-orbit-semigr-lemma}.
\end{proof}

\subsection{The shrinking method} Using the First Fundamental
Theorem for general linear groups, it is often possible to shrink
paths to just one arrow and still preserve weight spaces of
semi-invariants. The following shrinking procedure appears in some
form in \cite{DW1}, \cite{DL}, \cite{SW2}, and \cite{SW1}. We
include a proof since it is straight forward.

\begin{lemma}\label{reductionlemma.swo}
Let $Q$ be a quiver and $v_0$ a vertex such that near $v_0,$ $Q$
looks like:
$$
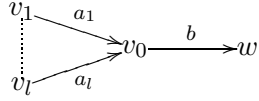
\xy     (0,0)*{v_0}="a";
        (-15,5)*{v_1}="b";
        (-15,-5)*{v_l}="c";
        (15,0)*{w}="d";
        {\ar^{a_1} "b";"a"};
        {\ar_{a_l} "c";"a"};
        {\ar^{b} "a";"d"};
        {\ar@{.} "b";"c"};
\endxy
$$

Suppose that $\beta$ is a dimension vector and $\sigma$ is a
weight such that $$\beta(v_0)\geq \beta(w)
\text{~and~}\sigma(v_0)=0.$$ Let $\overline Q$ be the quiver
defined by $\overline Q_0 = Q_0 \setminus \{ v_0 \}$ and
$\overline Q_1 = (Q_1\setminus \{b,a_1, \dots, a_l\}) \cup
\{ba_1,\dots,ba_l \}$. If $\overline \beta=\beta|_{\overline Q}$
and $\overline\sigma=\sigma|_{\overline Q}$ are the restrictions
of $\beta$ and $\sigma$ to $\overline Q$ then
$$ \SI(Q,\beta)_{\sigma}\cong \SI(\overline
Q,\overline\beta)_{\overline\sigma}.
$$
\end{lemma}

\begin{proof}
Consider the reduction map
$$
\pi:\Rep(Q,\beta)\to \Rep(\overline Q,\overline\beta),
$$
defined by taking compositions of linear maps along the new arrows
of $\overline Q$. As $\beta(v_0) \geq \beta(w)$, we know that
$\pi$ is a surjective morphism and hence the induced comorphism
$\pi^{\star}$ is injective. Using the First Fundamental Theorem
for $\GL(\beta(v_0))$ (see for example \cite{DK}), we obtain
$$\pi^{\star}(k[\Rep(\overline
Q,\overline\beta)])=k[\Rep(Q,\beta)]^{\GL(\beta(v_0))}.$$ Since
$\pi^{\star}$ is a $\GL(\overline\beta)$-equivariant surjective
linear map, %from $K[\Rep(Q,\overline\beta)]$ onto
%$K[\Rep(Q,\beta)]^{\GL(\beta(v_0))}$,
$\GL(\overline\beta)$ is linearly reductive, and $\sigma(v_0)=0,$
we have that $\pi^{\star}$ induces a surjective map from
$\SI(\overline Q,\overline\beta)_{\overline\sigma}$ onto
$\SI(Q,\beta)_{\sigma}$. (Note that at this point we need the
assumption that the base field is of characteristic zero.) So,
$\pi^{\star}$ defines an isomorphism from $\SI(\overline
Q,\overline\beta)_{\overline\sigma}$ into $\SI(Q,\beta)_{\sigma}$
and we are done.
\end{proof}

\begin{remark} Note that the lemma above remains true when we reverse the orientation of the arrows $a_i$ and
$b$.
\end{remark}

Keeping the same notations as above, we have:

\begin{proposition}\label{tool-shrink} If $Q$ satisfies property $(S)$ then so does $\overline
Q$.
\end{proposition}

\subsection{Exceptional sequences} A dimension vector $\beta$ is
called a \emph{real Schur} root if there exists a representation
$W \in \Rep(Q,\beta)$ such that $\End_Q(W) \simeq k$ and
$\Ext_{Q}^1(W,W)=0$ (we call such a representation
\emph{exceptional}). Note that if $\beta$ is a real Schur root
then there exists a unique, up to isomorphism, exceptional
$\beta$-dimensional representation.

For $\alpha$ and $\beta$ two dimension vectors, consider the
generic $\ext$ and $\hom$:
$$\ext_Q(\alpha,\beta)=\min \{\dim_k \Ext^1_Q(V,W) \mid (V,W)\in \Rep(Q,\alpha)\times \Rep(Q,\beta)\},$$
and
$$
\hom_Q(\alpha,\beta)=\min \{\dim_k \Hom_Q(V,W) \mid (V,W)\in
\Rep(Q,\alpha)\times \Rep(Q,\beta)\}.
$$
Given two dimension vectors $\alpha$ and $\beta$, we write $\alpha
\perp \beta$ provided that
$\ext_Q(\alpha,\beta)=\hom_Q(\alpha,\beta)=0$. If $W$ is a representation, we define ${}^\perp W$ to be the full
subcategory of $\Rep(Q)$ consisting of all representations $V$
such that $V \perp W$.

\begin{definition} We say that $(\varepsilon_1,\dots,\varepsilon_r)$ is an \emph{exceptional
sequence} if
\begin{enumerate}
\renewcommand{\theenumi}{\arabic{enumi}}

\item $\varepsilon_i$ are real Schur roots;

\item $\varepsilon_i \perp \varepsilon_j$ for all $1 \leq i<j \leq
l$.

\end{enumerate}
\end{definition}

Following \cite{DW2}, a sequence
$(\varepsilon_1,\dots,\varepsilon_r)$ is called a \emph{quiver
exceptional sequence} if it is exceptional and $\langle
\varepsilon_j,\varepsilon_i \rangle \leq 0$ for all $1 \leq i<j
\leq l$.

A sequence $(E_1, \dots, E_r)$ of exceptional representations is
said to be a (quiver) exceptional sequence if the sequence of
their dimension vectors $(\underline{d}_{E_1},
\dots,\underline{d}_{E_r})$ is a (quiver) exceptional sequence.

Now, let $\varepsilon=(\varepsilon_1,\dots,\varepsilon_r)$ be a
quiver exceptional sequence and let $E_i \in
\Rep(Q,\varepsilon_i)$ be exceptional representations. Construct a
new quiver $Q(\varepsilon)$ with vertex set $\{1,\dots,r\}$ and
$-\langle \varepsilon_j,\varepsilon_i \rangle$ arrows from $j$ to
$i$. Define $\mathcal{C}(\varepsilon)$ to be the smallest full
subcategory of $\Rep(Q)$ which contains $E_1, \dots, E_r$ and is
closed under extensions, kernels of epimorphisms, and cokernels of
monomorphisms.

For the remaining of this section, we assume that $r \leq N-1$,
where $N$ is the number of vertices of $Q$. We recall a very
useful result from \cite[Section 2.7]{DW2} in a form that is
convenient for us:
\begin{proposition}\label{excep-prop}\cite{DW2} The category $\mathcal{C}(\varepsilon)$ is naturally equivalent
to $\Rep(Q(\varepsilon))$ with $E_1, \dots, E_r$ being the simple
objects of $\mathcal{C}(\varepsilon)$. Furthermore, the inverse
functor from $\Rep(Q(\varepsilon))$ to $\mathcal{C}(\varepsilon)$
is a full exact embedding into $\Rep(Q)$.
\end{proposition}

\begin{proof} As the $\varepsilon_i$ are Schur roots and
$\ext_{Q}(\varepsilon_i,\varepsilon_j)=0,
\forall i <j$, we know that either
$\hom_{Q}(\varepsilon_j,\varepsilon_i)=0$ or
$\ext_{Q}(\varepsilon_j,\varepsilon_i)=0$ by \cite[Theorem
4.1]{S1}. This fact combined with $\langle
\varepsilon_j,\varepsilon_i \rangle \leq 0$ implies
$\hom_{Q}(\varepsilon_j,\varepsilon_i)=0$. But this is equivalent
to $\Hom_{Q}(E_j,E_i)=0, \forall i <j$ as the representations
$E_i$ have dense orbits in $\Rep(Q,\varepsilon_i)$. From
\cite[Lemma 2.36]{DW2}, it follows that $E_1, \dots, E_r$ are
precisely the simple objects of $\mathcal{C}(\varepsilon)$.

Using \cite[Theorem 2.3]{S2} (see also \cite{DW2}), we can
(naturally) extend $E_1, \dots, E_r$ to an exceptional sequence of
representations of the form $E_1, \dots, E_r,E_{r+1},\dots, E_N$.
Now, we have the equality
$$\mathcal{C}(\varepsilon)={}^\perp E_{r+1} \cap \dots \cap {}^\perp E_N,$$
by \cite[Lemma 4]{CB2}. Applying \cite[Theorem 2.3]{S2} again, we
deduce that $\mathcal{C}(\varepsilon)$ is naturally equivalent to
the category of representations of a quiver $Q'$ without oriented
cycles and $r$ vertices. Furthermore, the inverse functor from
$\Rep(Q')$ to $\mathcal{C}(\varepsilon)$ is a full exact embedding
into $\Rep(Q)$. But now, it is clear that $Q'$ must be precisely
$Q(\varepsilon)$.
\end{proof}

\begin{remark} %We should point out that in case $\varepsilon=(\varepsilon_1,\dots,\varepsilon_N)$ is
%a complete exceptional sequence, Crawley-Boevey \cite{CB2} proved
%that $\mathcal{C}(\varepsilon)=\Rep(Q)$. Moreover,
Let us point out that for a complete quiver exceptional sequence $\varepsilon$, the corresponding exceptional
representations $E_i$ are precisely the simple
representations of $Q$ as it was shown by Ringel \cite[Theorem
3]{R2}.
\end{remark}

From Proposition \ref{orbit-semi-def-prop} and Proposition
\ref{excep-prop}, we deduce:

\begin{proposition}\label{tool-exceptional} Let
$\varepsilon=(\varepsilon_1,\dots,\varepsilon_r)$ be a quiver
exceptional sequence for $Q$. If property $(S)$ fails for
$Q(\varepsilon)$ then the same is true for $Q$.
\end{proposition}

\section{Wild quivers}\label{wild-sec} In this section, we show that for every wild quiver $Q$ without
oriented cycles there is a representation $W$ such that $S(W)$ is
not saturated. Our strategy is very similar to the one from
\cite[Section 6]{SW1}. More precisely, we use reflection functors
and the shrinking method to reduce the list of wild quivers to
just seven quivers. Then, we use exceptional sequences to further
reduce this list to the generalized Kronecker quiver with three
arrows.

\begin{example}[Generalized Kronecker quivers]\label{main-cex} Let $\theta(3)$ be the generalized Kronecker
quiver with two vertices and three arrows:
$$
\xy     (0,0)*{\cdot}="a";
        (10,0)*{\cdot}="b";
        {\ar@3{->} "a";"b" };
\endxy
$$

Label the three arrows by $a,b$, and $c$. Now, choose $W \in
\Rep(\theta(3),(3,3))$ so that $W(a),W(b)$, and $W(c)$ are
linearly independent skew-symmetric $3 \times 3$ matrices. If
$\sigma=(1,-1)$ then it is easy to see that $\sigma(\underline
d_{W'})\leq 0$ for every subrepresentation $W'$ of $W$ (i.e., $W$
is $\sigma$-semi-stable). By King semi-stability criterion
\cite{K}, this is means that $n \sigma \in S(W)$ for some integer
$n \geq 1$.

On the other hand, we claim that $\sigma \notin S(W)$. Indeed, the
weight space $\SI(\theta(3),(3,3))_{\sigma}$ is spanned by the
coefficients of the functional determinant:
$$
W \to \det(t_1W(a)+t_2W(b)+t_3W(c))
$$
as a polynomial in the variables $t_1,t_2$, and $t_3$. But $3
\times 3$ skew symmetric matrices have zero determinant, and
hence, $f(W)=0$ for every semi-invariant $f$ of weight $\sigma$.
This shows that $S(W)$ is not saturated.
\end{example}

The following combinatorial result is essentially taken from
\cite[Proposition 49]{SW1} (see also \cite[Lemma 2.1, pp.
253]{ASS}):

\begin{proposition}\label{7-quivers} Let $Q$ be a finite, connected, wild quiver
without oriented cycles. Then $Q$ contains a subquiver which can
be reduced to one of the following seven quivers by applying
reflection transformations and shrinking paths to arrows:
\begin{enumerate}
\renewcommand{\theenumi}{\alph{enumi}}

\item

$$
\xy     (0,0)*{\cdot}="a";
        (10,0)*{\cdot}="b";
        {\ar@3{->} "a";"b" };
\endxy
$$

\item
$$
\xy     (0,0)*{\cdot}="a";
        (10,0)*{\cdot}="b";
        (20,0)*{\cdot}="c";
        {\ar@2{->} "a";"b"};
        {\ar@{->} "b";"c"};
\endxy
$$
\item
$$
\xy     (-10, 0)*{\cdot}="b";
        (0,0)*{\cdot}="c";
        (10,5)*{\cdot}="d";
        (-10,5)*{\cdot}="e";
        (-10,-5)*{\cdot}="f";
        (10,-5)*{\cdot}="g";
        {\ar@{->} "b";"c"};
        {\ar@{<-} "c";"d"};
        {\ar@{<-} "c";"e"};
        {\ar@{<-} "c";"f"};
        {\ar@{<-} "c";"g"};
\endxy
$$
\item
$$
\xy     (0,0)*{\cdot}="a";
        (10,5)*{\cdot}="b";
        (10,0)*{\cdot}="c";
        (10,-5)*{\cdot}="d";
        (-10,0)*{\cdot}="e";
        (-20,0)*{\cdot}="f";
        {\ar@{->} "a";"b"};
        {\ar@{->} "a";"c"};
        {\ar@{->} "a";"d"};
        {\ar@{->} "e";"a"};
        {\ar@{->} "f";"e"};
\endxy
$$
\item
$$
\xy (0, 0)*{\cdot}="a";
        (-10, 5)*{\cdot}="b";
        (-20,5)*{\cdot}="c";
        (-30,5)*{\cdot}="d";
        (-20,0)*{\cdot}="e";
        (-10,0)*{\cdot}="f";
        (-20,-5)*{\cdot}="h";
        (-10,-5)*{\cdot}="i";
        {\ar@{<-} "a";"b"};
        {\ar@{<-} "b";"c"};
        {\ar@{<-} "c";"d"};
        {\ar@{->} "f";"a"};
        {\ar@{->} "e";"f"};
        {\ar@{<-} "a";"i"};
        {\ar@{<-} "i";"h"};
    \endxy
$$
\item
$$
\xy (0, 0)*{\cdot}="a";
        (-10, 5)*{\cdot}="b";
        (-20,5)*{\cdot}="c";
        (-30,5)*{\cdot}="d";
        (-40,5)*{\cdot}="h";
        (-20,0)*{\cdot}="e";
        (-10,0)*{\cdot}="f";
        (-30,0)*{\cdot}="i";
        (-10,-5)*{\cdot}="j";
        {\ar@{->} "b";"a"};
        {\ar@{->} "c";"b"};
        {\ar@{->} "d";"c"};
        {\ar@{->} "h";"d"};
        {\ar@{->} "f";"a"};
        {\ar@{->} "e";"f"};
        {\ar@{->} "i";"e"};
        {\ar@{->} "j";"a"};
    \endxy
$$
\item
$$
\xy (0, 0)*{\cdot}="a";
        (-10, 5)*{\cdot}="b";
        (-20,5)*{\cdot}="c";
        (-30,5)*{\cdot}="d";
        (-40,5)*{\cdot}="h";
        (-50,5)*{\cdot}="g";
        (-60,5)*{\cdot}="k";
        (-20,0)*{\cdot}="e";
        (-10,0)*{\cdot}="f";
        (-10,-5)*{\cdot}="i";
        {\ar@{<-} "a";"b"};
        {\ar@{<-} "b";"c"};
        {\ar@{<-} "c";"d"};
        {\ar@{<-} "d";"h"};
        {\ar@{<-} "h";"g"};
        {\ar@{->} "k";"g"};
        {\ar@{->} "f";"a"};
        {\ar@{->} "e";"f"};
        {\ar@{<-} "a";"i"};
    \endxy
$$
\end{enumerate}
\end{proposition}

\begin{proof} Similar to \cite[Proposition 49]{SW1}.
\end{proof}

\begin{remark} Note that the first two quivers on our list differ
than the first two from \cite[Proposition 49]{SW1}. This is
because we are shrinking paths to arrows instead of shrinking
arrows to loops (or identity) as in the aforementioned paper.
Also, our subquiver $Q'$ is not necessarily a full subquiver.
\end{remark}

\begin{proposition}\label{ex-7-quivers-prop} Let $Q$ be one of the quivers from the list
obtained in Proposition \ref{7-quivers}. Then, there exists a
representation $W \in \Rep(Q,\beta)$ such that $S(W)$ is not
saturated.
\end{proposition}

\begin{proof} We have seen in Example \ref{main-cex} that our proposition is true for the quiver of type $(a)$.
Next, we use exceptional sequences to embed this generalized
Kronecker quiver into each of the remaining six quivers.

For the quiver
$$
\xy     (0,0)*{\cdot}="a";
        (10,0)*{\cdot}="b";
        (20,0)*{\cdot}="c";
        {\ar@2{->} "a";"b"};
        {\ar@{->} "b";"c"};
\endxy
$$
of type $(b)$, we take $\varepsilon_1=(0,0,1)$ and
$\varepsilon_2=(2,3,0)$.

For the quiver
$$
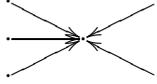
\xy     (-10, 0)*{\cdot}="b";
        (0,0)*{\cdot}="c";
        (10,5)*{\cdot}="d";
        (-10,5)*{\cdot}="e";
        (-10,-5)*{\cdot}="f";
        (10,-5)*{\cdot}="g";
        {\ar@{->} "b";"c"};
        {\ar@{<-} "c";"d"};
        {\ar@{<-} "c";"e"};
        {\ar@{<-} "c";"f"};
        {\ar@{<-} "c";"g"};
\endxy
$$
of type $(c)$, we take $\varepsilon_1=
\begin{matrix}
1& &1\\
1&3 & \\
1& & 0\\
\end{matrix}
$, \hspace{10pt}$\varepsilon_2=
\begin{matrix}
0& &0\\
0&0& \\
0& &1\\
\end{matrix}
$.

For the quiver
$$
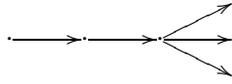
\xy     (0,0)*{\cdot}="a";
        (10,5)*{\cdot}="b";
        (10,0)*{\cdot}="c";
        (10,-5)*{\cdot}="d";
        (-10,0)*{\cdot}="e";
        (-20,0)*{\cdot}="f";
        {\ar@{->} "a";"b"};
        {\ar@{->} "a";"c"};
        {\ar@{->} "a";"d"};
        {\ar@{->} "e";"a"};
        {\ar@{->} "f";"e"};

\endxy
$$
of type $(d)$, take $\varepsilon_1=
\begin{matrix}
& & &0\\
0&0 &0 &0\\
& & & 1\\
\end{matrix}
$,\hspace{10pt} $\varepsilon_2=
\begin{matrix}
& & &1\\
0&1 &2 &1\\
& & & 0\\
\end{matrix}
$, \hspace{10pt}$\varepsilon_3=
\begin{matrix}
&&&0\\
1&0&0&0\\
&&&0
\end{matrix}
$.
Note that for this quiver, the generalized Kronecker quiver embeds via the quiver of type $(b)$.

For the quiver
$$
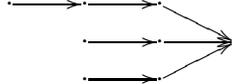
\xy (0, 0)*{\cdot}="a";
        (-10, 5)*{\cdot}="b";
        (-20,5)*{\cdot}="c";
        (-30,5)*{\cdot}="d";
        (-20,0)*{\cdot}="e";
        (-10,0)*{\cdot}="f";
        (-20,-5)*{\cdot}="h";
        (-10,-5)*{\cdot}="i";
        {\ar@{<-} "a";"b"};
        {\ar@{<-} "b";"c"};
        {\ar@{<-} "c";"d"};
        {\ar@{->} "f";"a"};
        {\ar@{->} "e";"f"};
        {\ar@{<-} "a";"i"};
        {\ar@{<-} "i";"h"};
    \endxy
$$
of type $(e)$, we take $\varepsilon_1=
\begin{matrix}
0&1&2&\\
&1&2&3\\
&0&2& \\
\end{matrix}
$,\hspace{10pt} $\varepsilon_2=
\begin{matrix}
0&0&0&\\
&0&0&0\\
&1&0& \\
\end{matrix}
$, \hspace{10pt}$\varepsilon_3=
\begin{matrix}
1&0&0&\\
&0&0&0\\
&0&0&
\end{matrix}
$.
Note that for this quiver, first we get an embedding of the quiver $
\xy     (0,0)*{\cdot}="a";
        (10,0)*{\cdot}="b";
        (20,0)*{\cdot}="c";
        {\ar@2{->} "a";"b"};
        {\ar@{<-} "b";"c"};
\endxy
$ (call it of type $(b')$) and then we embed the generalized Kronecker
quiver into the quiver of type $(b')$ by using the sequence $((2,3,0),(0,0,1))$.

For the quiver
$$
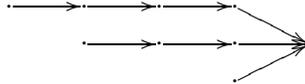
\xy (0, 0)*{\cdot}="a";
        (-10, 5)*{\cdot}="b";
        (-20,5)*{\cdot}="c";
        (-30,5)*{\cdot}="d";
        (-40,5)*{\cdot}="h";
        (-20,0)*{\cdot}="e";
        (-10,0)*{\cdot}="f";
        (-30,0)*{\cdot}="i";
        (-10,-5)*{\cdot}="j";
        {\ar@{->} "b";"a"};
        {\ar@{->} "c";"b"};
        {\ar@{->} "d";"c"};
        {\ar@{->} "h";"d"};
        {\ar@{->} "f";"a"};
        {\ar@{->} "e";"f"};
        {\ar@{->} "i";"e"};
        {\ar@{->} "j";"a"};
    \endxy
$$
of type $(f)$, we take $\varepsilon_1=
\begin{matrix}
0&3&5&7&\\
&3&5&7&9\\
& & &5& \\
\end{matrix}
$,\hspace{10pt} $\varepsilon_2=
\begin{matrix}
1&0&0&0&\\
&0&0&0&0\\
& & &0& \\
\end{matrix}
$.

Finally, for the quiver
$$
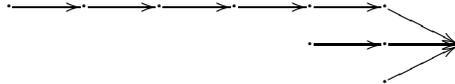
\xy (0, 0)*{\cdot}="a";
        (-10, 5)*{\cdot}="b";
        (-20,5)*{\cdot}="c";
        (-30,5)*{\cdot}="d";
        (-40,5)*{\cdot}="h";
        (-50,5)*{\cdot}="g";
        (-60,5)*{\cdot}="k";
        (-20,0)*{\cdot}="e";
        (-10,0)*{\cdot}="f";
        (-10,-5)*{\cdot}="i";
        {\ar@{<-} "a";"b"};
        {\ar@{<-} "b";"c"};
        {\ar@{<-} "c";"d"};
        {\ar@{<-} "d";"h"};
        {\ar@{<-} "h";"g"};
        {\ar@{->} "k";"g"};
        {\ar@{->} "f";"a"};
        {\ar@{->} "e";"f"};
        {\ar@{<-} "a";"i"};
    \endxy
$$
of type $(g)$, we take $\varepsilon_1=
\begin{matrix}
0&3&5&7&9&11&\\
&&&&5&9&13\\
&&&&&7& \\
\end{matrix}
$,\hspace{10pt} $\varepsilon_2=
\begin{matrix}
1&0&0&0&0&0&\\
&&&&0&0&0\\
&&&&&0& \\
\end{matrix}
$.

Now, the proof follows from Proposition \ref{tool-exceptional}.
\end{proof}

\begin{proof}[Proof of Theorem \ref{main-thm}] The implication
$"\Longrightarrow"$ is proved in Proposition
\ref{tamequivers-prop}. The other implication follows from
Proposition \ref{ex-7-quivers-prop} and Proposition
\ref{tool-reflection}, Proposition \ref{tool-shrink}, Proposition
\ref{tool-exceptional}.
\end{proof}

\begin{remark}We would like to end this section with some comments about the possibility of extending our theorem to other classes of algebras. First of all, it is obvious how to define orbit semigroups for finite dimensional modules over finite dimensional algebras. Furthermore, some of the main tools used in the proof of Theorem \ref{main-thm}, such as Derksen-Weyman
spanning theorem and Kac's canonical decomposition, are available
for finite dimensional algebras as well. It is also useful to
know if Schofield's theorem \cite[Theorem 3.8]{S1} extends to
other classes of algebras. This is clearly the case for regular
dimension vectors for concealed-canonical algebras (for more
details, see \cite[pp. 382]{DL2}). This opens the possibility of
proving our theorem for this class of algebras. Finally, let us
mention that for canonical algebras, the implication
$"\Longleftarrow"$ of Theorem \ref{main-thm} follows from its
validity for quivers. Indeed, from \cite{R3} we know that a canonical algebra $\Lambda$ with underlying quiver $\Delta$ is tame if and only if $\Delta \setminus \{\infty \}$ is a Dynkin or Euclidean quiver. (Here, $\infty$ is the unique sink of $\Delta$.) Now, we can see that by working with representations of $\Lambda$ which are zero at the sink $\infty$, the implication $"\Longleftarrow"$ of Theorem \ref{main-thm} for canonical algebras follows.
\end{remark}

\section{The thin sincere case} In this section we look into the
case when the dimension vector $\textbf{1}$ is thin sincere, i.e.,
$\textbf{1}(x)=1, \forall x \in Q_0$. Let us fix some notation first. For an affine $G$-variety $X$,
where $G$ is a linear algebraic group, and $\sigma \in X^{*}(G)$ a
rational character of $G$, we set
$$
\SI(X,G)_{\sigma}:=\{f \in k[X] \mid g \cdot f=\sigma(g)f, \forall
g \in G\}.
$$
For a given representation $W \in \Rep(Q, \textbf{1})$, we have
$$
S(W)=\{\sigma \in \ZZ^{Q_0} \mid \exists f \in
\SI(\overline{\GL(\textbf{1})W},\GL(\textbf{1}))_{\sigma} \text{~such that~}
f(W)\neq 0 \}.
$$
Consider the weight space decomposition:
$$k[\overline{\GL(\textbf{1})W}]=k[\overline{\GL(\textbf{1})W}]^{\SL(\textbf{1})}=\bigoplus
\SI(\overline{\GL(\textbf{1})W},\GL(\textbf{1}))_{\sigma},$$ where the sum
is over all weights $\sigma \in S(W)$. As $\GL(\textbf{1})$ acts with
a dense orbit on the closure of the orbit of $W$, we must have
$\dim_k \SI(\overline{\GL(\textbf{1})W},\GL(\textbf{1}))_{\sigma} \leq 1$, and
so,
$$k[\overline{\GL(\textbf{1})W}]=k[S(W)].$$
From toric geometry, we deduce that if $S(W)$ is saturated then $\overline{\GL(\textbf{1})W}$ is normal. We should point out that this last observation is not true for other dimension vectors as the following example, due to Zwara, shows.

\begin{example} Consider the Kronecker
quiver $\theta(2)$:
$$
\xy     (0,0)*{1}="a";
        (10,0)*{2}="b";
        {\ar@2{->} "a";"b" };
\endxy
$$
with arrows labeled by $a$ and $b$. Let $W \in \Rep(\theta(2),(3,3))$ be the representation given by $W(a)=
\left(
\begin{matrix}
0& 0 & 0\\
1& 0 & 0\\
0& 1 & 0
\end{matrix}
\right)$ and $W(b)= \left(
\begin{matrix}
1& 0 & 0\\
0& 0 & 0\\
0& 0 & 1
\end{matrix}
\right)$. It was proved by Zwara \cite{Zw2} that $\overline{\GL(\alpha)W}$ is not normal. On the other hand, Proposition \ref{tamequivers-prop} tells us that $S(W)$ is saturated. In fact, in this example it is not difficult to see that $S(W)=\{(0,0)\}$. Indeed, it was first proved by Happel \cite{Hap} (see also \cite{Ko}) that the algebra of semi-invariants $\SI(\theta(2),(3,3))$ is (a polynomial algebra) generated by the
coefficients of the functional determinant:
$$
W \to \det(t_1W(a)+t_2W(b))
$$
as a polynomial in the variables $t_1,t_2$. But, for our choice of $W(a)$ and $W(b)$, $\det(t_1W(a)+t_2W(b))=0$ and hence $S(W)=\{(0,0)\}$.
\end{example}

In \cite[Theorem 1.3]{BZ2}, it was proved that $\overline{\GL(\textbf{1})W}$ is normal
when $W \in \Rep(Q,\textbf{1})$ is just the identity along the
arrows. We are going to see that this is the case for any
representation $W$ by showing that $S(W)$ is saturated. In the thin sincere case, it is rather easy to
write down a $k$-basis for each weight space of semi-invariants.
Consider the boundary map $I=I_Q:\RR^{Q_1} \to \RR^{Q_0}$ of $Q$;
this is the function which assigns to every
$\lambda=(\lambda(a))_{a \in Q_1}$, the vector $(I(\lambda)_x)_{x
\in Q_0}$, where
$$
I(\lambda)_{x}:=\sum_{\buildrel {a \in Q_1} \over
{ta=x}}\lambda(a)- \sum_{\buildrel{a \in Q_1} \over
{ha=x}}\lambda(a),
$$
for every vertex $x \in Q_0$. Let us record the following simple
lemma:

\begin{lemma} Keep notation as above. Let $\sigma \in \ZZ^{Q_0}$
be a weight. Then
$$
\dim_k \SI (Q, \emph{\textbf{1}})_{\sigma}=|I^{-1}(\sigma) \bigcap
\ZZ^{Q_1}_{\geq 0}|.
$$
\end{lemma}

\begin{proof} For convenience, denote $V(x)=k, \forall x \in Q_0$.
If $V$ is a vector space, we denote by $\det_V^l$, the $l^{th}$
power of the determinant representation of $\GL(V)$; the symmetric
algebra of $V$ is denoted by $S(V)$. It is easy to see that
$$
\begin{aligned}
k[\Rep(Q,\textbf{1})]=&\bigotimes_{a \in Q_1}S(V(ta) \otimes
V(ha)^*)\\
&=\bigotimes_{a \in Q_1} \bigoplus_{\lambda(a)\in \ZZ_{\geq 0}}
{\det}_{V(ta)}^{\lambda(a)} \otimes {\det}_{V(ha)}^{-\lambda(a)}\\
&=\bigoplus_{\lambda \in \ZZ_{\geq 0}^{Q_1}}\bigotimes_{x \in Q_0}
{\det}_{V(x)}^{I(\lambda)_x}
\end{aligned}
$$
The proof now follows.
\end{proof}

For every $\lambda \in \ZZ^{Q_1}_{\geq 0}$, define
$$
\begin{aligned}
f_{\lambda}: &\Rep(Q, \textbf{1}) \to k\\
&(t(a))_{a \in Q_1} \mapsto \prod_{a \in Q_1} t(a)^{\lambda(a)}.
\end{aligned}
$$
Now it is clear that $\{f_{\lambda} \mid \lambda \in
I^{-1}(\sigma) \bigcap \ZZ^{Q_1}_{\geq 0}\}$ is a $k$-basis of
$\SI(Q, \textbf{1})_{\sigma}$.

\begin{proposition} For every $W \in \Rep(Q, \emph{\textbf{1}})$, the
semigroup $S(W)$ is saturated.
\end{proposition}

\begin{proof} Write $W=(t(a))_{a \in Q_1}$, where $t(a) \in k, \forall a \in Q_1$. Note that for a weight $\sigma \in \ZZ^{Q_0}$, we have
$$
\sigma \in S(W) \Longleftrightarrow \exists \lambda \in
I^{-1}(\sigma) \bigcap \ZZ^{Q_1}_{\geq 0} \text{~such that~}
\lambda(a)=0 \text{~whenever~} t(a)=0.
$$

To check that $S(W)$ is saturated, we can clearly assume that
$t(a) \neq 0, \forall a \in Q_1$. Under this assumption, we deduce
that $S(W)=\Sigma(Q, \textbf{1})$ which is known to be saturated
by Theorem \ref{DW-sat}.
\end{proof}

%\bibliography{biblio}
\end{document}